\documentclass{amsart}

\usepackage{enumerate}
\usepackage{amsmath}%
\usepackage{amsfonts}%
\usepackage{amssymb}%
\usepackage{graphicx}
\usepackage{mathrsfs}
\usepackage{hyperref}
\usepackage{fullpage}
\newtheorem{theorem}{Theorem}
\theoremstyle{plain}

\newtheorem{claim}[theorem]{Claim}

\newtheorem{construction}[theorem]{Construction}

\newtheorem{conjecture}[theorem]{Conjecture}
\newtheorem{corollary}[theorem]{Corollary}

\newtheorem{lemma}[theorem]{Lemma}

\newtheorem{proposition}[theorem]{Proposition}

\numberwithin{equation}{section}
\numberwithin{theorem}{section}
\numberwithin{case}{section}

\numberwithin{subcase}{case}
\def\F{\mathcal{F}}

\def\cP{\mathcal{P}}

\def \a{\alpha}
\def \e{\epsilon}
\def \r{\gamma}

\def \bfi{\mathbf{i}}
\def \bfu{\mathbf{u}}
\def \bfv{\mathbf{v}}

\begin{document}
\title{Near Perfect Matchings in $k$-uniform Hypergraphs II}
\author{Jie Han}
\address{Instituto de Matem\'{a}tica e Estat\'{\i}stica, Universidade de S\~{a}o Paulo, Rua do Mat\~{a}o 1010, 05508-090, S\~{a}o Paulo, Brazil}
\email[Jie Han]{jhan@ime.usp.br}
\thanks{The author is supported by FAPESP (2014/18641-5, 2015/07869-8).}
\date{\today}

\begin{abstract}
Suppose $k\nmid n$ and $H$ is an $n$-vertex $k$-uniform hypergraph.
A near perfect matching in $H$ is a matching of size $\lfloor n/k\rfloor$.
We give a divisibility barrier construction that prevents the existence of near perfect matchings in $H$. This generalizes the divisibility barrier for perfect matchings.
We give a conjecture on the minimum $d$-degree threshold forcing a (near) perfect matching in $H$ which generalizes a well-known conjecture on perfect matchings. We also verify our conjecture for various cases. 
Our proof makes use of the lattice-based absorbing method that the author used recently
to solve two other problems on matching and tilings for hypergraphs.
\end{abstract}

\maketitle

\section{Introduction}

Given $k\ge 2$, a $k$-uniform hypergraph (in short, \emph{$k$-graph}) consists of a vertex set $V(H)$ and an edge set $E(H)\subseteq \binom{V(H)}{k}$, where every edge is a $k$-element subset of $V(H)$. A \emph{matching} in $H$ is a collection of vertex-disjoint edges of $H$. A \emph{perfect matching} $M$ in $H$ is a matching that covers all vertices of $H$. Clearly a perfect matching in $H$ exists only if $k$ divides $|V(H)|$. When $k$ does not divide $n=|V(H)|$, we call a matching $M$ in $H$ a \emph{near perfect matching} if $|M|=\lfloor n/k \rfloor$.

Given a $k$-graph $H$ with a set $S$ of $d$ vertices (where $1 \le d \le k-1$) we define $\deg_{H} (S)$ to be the number of edges containing $S$ (the subscript $H$ is omitted if it is clear from the context). The \emph{minimum $d$-degree $\delta _{d} (H)$} of $H$ is the minimum of $\deg_{H} (S)$ over all $d$-vertex sets $S$ in $H$.  
We refer to $\delta _{k-1} (H)$ as the \emph{minimum codegree} of $H$.

\subsection{Matchings in hypergraphs via degree conditions}
For integers $n, k, d, s$ such that $1\le d\le k-1$ and $0\le s\le n/k$, let $m_d^{s}(k,n)$ denote the smallest integer $m$ such that $\delta_d(H)\ge m$ forces the existence of a matching in $H$ of size $s$ for any $k$-graph $H$ on $n$ vertices.
Throughout this note, $o(1)$ represent a function of $n$ that tends to 0 when $n$ goes to infinity.
The following conjecture~\cite{HPS, KuOs_survey09} has received much attention in the last few years: codegree~\cite{KO06mat, RRS06mat, RRS09}, and for $1\le d\le k-2$, approximate $d$-degree~\cite{AFHRRS, HPS, MaRu, Pik}, exact $d$-degree~\cite{CzKa, Han15mat, Khan1, Khan2, KOT, TrZh12, TrZh13, TrZh15} (also see surveys~\cite{RR, zsurvey}).
%The following conjecture \cite{HPS, KuOs_survey09} has received much attention in the last few years \cite{AFHRRS, CzKa, HPS, Khan2, Khan1, KO06mat, KOT, MaRu, Pik, RR, RRS06mat, RRS09, TrZh12, TrZh13, TrZh15}.

\begin{conjecture}\label{conj:mat}
For $1\le d\le k-1$ and $k \mid n$,
\[
m_d^{n/k} (k,n) = \left(\max \left\{ \frac12, 1- \left(1- \frac{1}{k} \right)^{k-d}  \right\} +o(1)\right)\binom{n-d}{k-d}.
\]
\end{conjecture}

We remark that the quantities in the conjecture come from the so-called \emph{divisibility barrier} and the \emph{space barrier}. 

\begin{construction}[Space Barrier]\label{con:sb}
Let $V$ be a set of size $n$ and fix $S\subseteq V$ with $|S|=s < n/k$. Let $H(s)$ be the $k$-graph on $V$ whose edges are all $k$-sets that intersect $S$. It is easy to see that the size of a maximum matching in $H(s)$ is $s< n/k$ and $\delta_d(H(s)) = \binom{n-d}{k-d} - \binom{n-d-s}{k-d} = (1- (1- s/n)^{k-d}  +o(1))\binom{n-d}{k-d}$.
\end{construction}

Moreover, the maximal value of $\delta_d(H(s))$ is attained by $s = \lceil n/k \rceil - 1$, which gives the second term in Conjecture~\ref{conj:mat}.

\begin{construction}[Divisibility Barrier with two parts, \cite{TrZh12}]\label{con:db}
Fix integers $j, n$ such that $j\in \{0,1\}$ and $k \mid n$.
Let $V$ be a set of size $n$ with a partition $V_1\cup V_2$ such that $|V_1| \not\equiv j n/k$ mod 2. Let $H^j$ be the $k$-graph on $V$ whose edges are all $k$-sets $e$ such that $|e\cap V_1|\equiv j$ mod 2.
\end{construction}

To see why $H^j$ does not have a perfect matching, for any matching $M$ in $H^j$, by definition, $|V(M)\cap V_1|\equiv j n/k$ mod 2. This means $V(M)\cap V_1\neq V_1$ and thus $M$ is not perfect.
It is not hard to see that $\delta_d(H^j) \le (\frac12 +o(1)) \binom{n-d}{k-d}$ and the equality is attained when $|V_1|\approx |V_2|\approx n/2$.
Unfortunately, we do not know which exact values of $|V_1|$ and $|V_2|$ maximize the value of $\delta_d(H^j)$ -- this forms a challenging question, see \cite[Section 1]{TrZh12} for a discussion.

The general cases when $s<n/k$ have also been studied.
Note that in the general case we do not require that $k\mid n$.
By a simple greedy algorithm, R\"odl, Ruci\'nski and Szemer\'edi \cite{RRS09} proved that $m_{k-1}^s(k, n)= s$ for all $s\le \lfloor n/k \rfloor - (k-2)$. 
Recently, the author \cite{Han14_mat} extended this to all $s<n/k$, verifying a conjecture of R\"odl, Ruci\'nski and Szemer\'edi.
Note that this is the best one can do, because when $k\mid n$, $m_{k-1}^{n/k}(k, n)\approx n/2$ by the main result of \cite{RRS09}.
Moreover, K\"uhn, Osthus and Treglown \cite{KOT} determined $m_1^s(3,n)$ for all $s\le n/k$.
Bollob\'as, Daykin and Erd\H{o}s \cite{BDE76} determined $m_1^s(k,n)$ for $s< n/2k^3-1$.
Recently, K\"uhn, Osthus and Townsend \cite{KOTo} determined $m_d^s(k,n)$ asymptotically for $1\le d\le k-2$ and $s\le \min\{ n/2(k-d), (1-o(1))n/k \}$.

\subsection{Near perfect matchings}
In this paper we are interested in the case $s = \lfloor n/k \rfloor$. We write $m_d'(k, n): = m_d^{\lfloor n/k \rfloor} (k,n)$, namely, the minimum $d$-degree threshold forcing the existence of a (near) perfect matching.

To state our results, we need to define the threshold for almost perfect matchings. For integers $1\le d\le k-1$, let
\[
c^*_d(k) := \lim_{\a \rightarrow 0} \lim_{n\rightarrow \infty} \left(m_d^{\lfloor n/k - \a n \rfloor}(k,n)/\binom{n-d}{k-d}\right).
\]

The following theorem follows from \cite[Lemma 2.4]{HPS} and the definition of $c^*_d(k)$.
\begin{theorem}\cite{HPS} \label{thm:AFHRRS}
For $1\le d\le k-1$ and $k \mid n$, we have
\[
m_d^{n/k}(k,n) = \left(\max \left\{ 1/2, c^*_d(k) \right\} +o(1)\right)\binom{n-d}{k-d}.
\]
\end{theorem}

Let $\ell \equiv n$ mod $k$.
For any $k$-graph $H$ on $n$ vertices with $\delta_d(H)\ge \left(\max \left\{ 1/2, c^*_d(k) \right\} +o(1)\right)\binom{n-d}{k-d}$, we delete arbitrary $\ell$ vertices of $V(H)$ and thus $\delta_d(H)$ decreases at most $\ell\binom{n-d-1}{k-d-1}=o(\binom{n-d}{k-d})$. Then we can find a perfect matching by applying Theorem~\ref{thm:AFHRRS} and get a near perfect matching of $H$. This gives the following corollary.
The lower bound is by the definition of $c^*_d(k)$.
\begin{corollary} \label{cor:HPS}
For $1\le d\le k-1$, we have
\[
\left(c^*_d(k) +o(1)\right)\binom{n-d}{k-d}\le m_d'(k,n) \le \left(\max \left\{ 1/2, c^*_d(k) \right\} +o(1)\right)\binom{n-d}{k-d}.
\]
\end{corollary}

Corollary~\ref{cor:HPS} implies that if $c^*_d(k)\ge 1/2$, then $m_d'(k,n)=\left(c^*_d(k) +o(1)\right)\binom{n-d}{k-d}$, which 'coincides' the value of $m_d^{n/k}(k,n)$.
So it is interesting to study $m_d'(k,n)$ when $c^*_d(k)< 1/2$.
However, the value of $c^*_d(k)$ for $d< k/2$ is still wide open.
Indeed, Construction~\ref{con:sb} shows that $c^*_d(k) \ge 1 - (1 - 1/k)^{k-d}$.
The authors of \cite{KOTo} conjectured that $c^*_d(k) = 1 - (1 - 1/k)^{k-d}$ and verified the cases when $d\ge k/2$. 
Also, it is shown by the result of \cite{AFHRRS} that $c^*_d(k) = 1 - (1 - 1/k)^{k-d}$ for $1\le k-d\le 4$.
\begin{theorem}\cite{AFHRRS, KOTo} \label{thm:KOT}
For positive integers $k, d$ such that $k\ge 3$ and $\min\{k/2, k-4\}\le d\le k-1$, we have
\[
c^*_d(k)= 1- \left(1- \frac{1}{k} \right)^{k-d} .
\]
\end{theorem}

When $k\nmid n$, the knowledge of $m_d'(k,n)$ is very limited.
The aforementioned results \cite{Han14_mat, KOT} included $m_{k-1}'(k, n)= \lfloor n/k \rfloor$ and $m_1'(3,n) = \binom{n-1}2 - \binom{\lceil 2n/3 \rceil}2 +1 = (5/9+o(1))\binom {n-1}2$.
In this paper we give new upper and lower bounds on $m_d'(k,n)$ for $k\ge 4$ and $1\le d\le k-2$.
%Our main results are the following.

\begin{theorem}\label{thm:npm2}
For $k\ge 4$, $1\le d\le k-2$ and $k\nmid n$, we have
\[
\left(\max \left\{ g(k,d,\ell), c^*_d(k) \right\} +o(1)\right)\binom{n-d}{k-d}\le m_d'(k,n) \le \left(\max \left\{ g(k,d,1), c^*_d(k) \right\} +o(1)\right)\binom{n-d}{k-d}.
\]
\end{theorem}

The term $g(k,d,\ell)$ will be defined formally later.
Our second result concerns the case $d=k-2$ and $\ell\in \{2,\dots, k-1\}$.
Note that the approximate version of $k=3$ case was obtained in~\cite{HPS} and the exact vertex degree condition in~\cite{KOT}.
%Note that the $k=3$ case was obtained in \cite{KOT}.

\begin{theorem}\label{thm:npm3}
For integers $k,\ell,n$ such that $k\ge 4$ and $n\equiv \ell$ mod $k$ for some $\ell\in \{2,\dots, k-1\}$, we have
\[
m_{k-2}'(k,n) = \left(1- \left(1- \frac{1}{k} \right)^{2} +o(1)\right)\binom{n-k+2}{2}=  \left(\frac{2k-1}{k^2} +o(1)\right)\binom{n-k+2}{2}.
\]
\end{theorem}

By the previous results \cite{Han14_mat, KOTo}, it seems that $m_d'(k,n)$ is determined only by the space barrier. 
However, we give the following divisibility barrier construction which generalizes Construction~\ref{con:db}, showing that this is not always the case when $d\le k-2$.

\begin{construction}\label{con:db1}
Fix integers $\ell, j, n$ and a real number $0\le x\le 1$ such that $0\le \ell\le k-1$, $j\in \{0,1,\dots, \ell+1\}$, $n\equiv \ell$ mod $k$ and $xn \equiv \lfloor \frac nk \rfloor j + \ell + 1$ mod $\ell+2$.
Let $V$ be a set of size $n$ with a partition $V_1\cup V_2$ such that $|V_1| = xn$. Let $H_{\ell}^j(x)$ be the $k$-graph on $V$ whose edges are all $k$-sets $e$ such that $|e\cap V_1|\equiv j$ mod $\ell+2$.
\end{construction}

Let $M$ be a near perfect matching in $H_\ell^j(x)$ and thus $M$ leaves a set $U$ of exactly $\ell$ vertices uncovered. Suppose $|U\cap V_1| = i'$ for some $0\le i'\le \ell$, then we have $|V(M)\cap V_1| = |V_1| - i' \equiv \lfloor \frac nk \rfloor j + \ell + 1 - i'$ mod $\ell+2$. On the other hand, because $|e\cap V_1|\equiv j$ mod $\ell+2$ for every edge $e\in M$, we have $|V(M)\cap V_1| \equiv \lfloor \frac nk \rfloor j$ mod $\ell+2$. This is a contradiction because $1\le \ell+1-i'\le \ell+1$, and thus $H_\ell^j(x)$ contains no near perfect matching.

Similar as in Construction~\ref{con:db}, it is not clear which exact values of $|V_1|, |V_2|$ maximize 
$\delta_d(H_\ell^j(x))$ in Construction~\ref{con:db1}. 
Indeed, it seems even harder than its special case Construction~\ref{con:db} (when $\ell=0$) -- the maximum of $\delta_d(H_\ell^j(x))$ is not always achieved by the (almost) balanced bipartition, i.e., $x\approx 1/2$ (see Section 2).
For this reason we introduce the variable $x$ in Construction~\ref{con:db1} and focus on the major term by letting $|V_1| = x n$.
For $1\le d\le k-1$ and $0\le \ell\le k-1$, we define
\begin{equation*}%\label{eq:gkd}
g(k, d, \ell) := \lim_{n\rightarrow \infty} \left(\max_{0\le j\le \ell+1} \max_{0\le x\le 1} \delta_d(H_\ell^j(x))/\binom{n-d}{k-d}\right).
\end{equation*}
%Construction~\ref{con:db1} says that if $n\equiv \ell$ mod $k$, then $m_d'(k,n) \ge (g(k,d,\ell)+o(1))\binom{n-d}{k-d}$ and together with the definition of $c^*_d(k)$, this shows the lower bound in Theorem~\ref{thm:npm2}.
We make the following conjecture.
As mentioned in the previous section, $g(k,d,0) = 1/2$, so Conjecture~\ref{conj:mat} is a special case (when $\ell=0$) of Conjecture~\ref{conj:npm}.

\begin{conjecture}\label{conj:npm}
For all $1\le d\le k-1$ and $n\equiv \ell$ mod $k$ for some $0\le \ell \le k-1$,
\begin{equation}\label{eq:conj}
m_d'(k,n) = \left(\max \left\{ g(k,d,\ell), 1- \left(1- \frac{1}{k} \right)^{k-d} \right\} +o(1)\right)\binom{n-d}{k-d}.
\end{equation}
\end{conjecture}

Regarding the quantitative values of $g(k,d,\ell)$, we will show that $g(k,d,1)$ can be determined explicitly by a simple optimization (see Section 2) and we have the following bounds on $g(k,d,\ell)$.

\begin{proposition}\label{prop:gkdll}
Let $1\le d\le k-1$ and $0\le \ell \le k-1$. Then
\begin{itemize}
\item[(i)] When $d\le \ell+1$, $g(k,d,\ell)\le 1/(d+1)$ and when $d\ge \ell+1$, $g(k,d,\ell)\le 1/{(\ell+2)}$.
\item[(ii)] $\left\lfloor \frac{2^{k-d}}3 \right\rfloor 2^{d-k}\le g(k,d,1) <\frac13$ for $d>1$.
\end{itemize}
\end{proposition}

So when $d>1$ and $k-d$ tends to infinity, $g(k,d,1)$ tends to $1/3$. We believe that this is true in general, that is, when $k-d$ tends to infinity, $d>\ell$ and $\ell$ stays as a constant, $g(k,d,\ell)$ tends to $1/(\ell+2)$. 
Note that for $d=k-1$, Proposition~\ref{prop:gkdll}(ii) says that $0\le g(k,k-1,1)<1/3$. In fact we have $g(k,k-1,\ell)=0$ for all $1\le \ell\le k-1$.
This is shown in the following proposition.

\begin{proposition}\label{prop:gkd}
It holds that $g(k, d, \ell) = 0$ for $d\ge \max\{k-\ell, \lfloor k/2 \rfloor +1\}$.
\end{proposition}

This explains the result in \cite{Han14_mat} -- when $d=k-1$, we have $d\ge \max\{k-\ell, \lfloor k/2 \rfloor +1\}$ unless $\ell=0$ (so Conjecture~\ref{conj:npm} is true for $d=k-1$).
%We will prove these two propositions in Section 2.
We obtain following remarks by Propositions~\ref{prop:gkdll} and~\ref{prop:gkd}.
%\noindent\textbf{Remark.}
\begin{enumerate}
\item
Note that $1 - (1 - 1/k)^{k-d} \ge g(k,d,1)$ if $d\le (1-\ln (3/2))k\approx 0.59k$. Indeed, since $(1-1/k)^k < 1/e$, we have
\[
\left(1- \frac{1}{k} \right)^{k-d} < \left(\frac 1e \right)^{\frac{k-d}k} \le \left(\frac 1e \right)^{\ln (3/2)} = \frac23.
\]
So $1 - (1 - 1/k)^{k-d}\ge 1/3 > g(k,d,1)$ for $d\ge 2$ and $1 - (1 - 1/k)^{k-1}> 1/2 > g(k,1,1)$ by Proposition~\ref{prop:gkdll}.
%This means that \eqref{eq:conj} holds for $k\ge 4$, $d\in [0.5k, 0.59k]$, $\ell\in [k-1]$ and sufficiently large $n$.
%\item
%Also note that

\item
When $2\le \ell \le k-1$, $d=k-2$ and $k\ge 4$, by Proposition~\ref{prop:gkd}, we have $g(k,k-2,\ell)=0$. 
So Theorem~\ref{thm:npm3} confirms Conjecture~\ref{conj:npm} for this range.
%, i.e., $m_{k-2}'(k,n)$ is given by the space barrier.

\end{enumerate}

These remarks, together with Theorems~\ref{thm:KOT} -- ~\ref{thm:npm3}, imply Conjecture~\ref{conj:npm} for various cases. We summarize them in the following corollary.
Note that by the remarks, $g(k,d,\ell)$ is always smaller than $c^*_d(k)$ in the following corollary, except the case $\ell=1$.

\begin{corollary}\label{cor:npm}
Fix positive integers $k,d,n$ such that $k\ge 4$, $\min\{k/2, k-4\}\le d\le k-2$ and $n \equiv \ell$ mod $k$ with $\ell \in [k-1]$.
Then \eqref{eq:conj} holds if at least one of the following holds.
\begin{enumerate}[$(1)$]
\item $\ell=1$;

\item
$d\le (1-\ln (3/2))k\approx 0.59k$. 

\item
$d=k-2$.
\end{enumerate}
%If $ 1 - (1 - 1/k)^{k-d}\ge g(k,d,1)$ or $\ell =1$, then \eqref{eq:conj} holds.
\end{corollary}

\subsection{Lattice-based absorbing method}

We use the \emph{lattice-based absorbing method} in the proof of Theorem~\ref{thm:npm2}.
The absorbing technique initiated  by R\"odl, Ruci\'nski and Szemer\'edi \cite{RRS06} has been shown to be efficient on finding spanning structures in graphs and hypergraphs. Roughly speaking, the goal is to build the \emph{absorbing set}, which is a small subset of vertices and can be used to `absorb' another small set of arbitrary vertices.
For finding perfect matchings, the so-called \emph{Strong Absorbing Lemma} \cite[Lemma 2.4]{HPS} says that when the minimum degree condition guarantees that every two vertices are \emph{reachable}, we can find the absorbing set in the hypergraph.

So the question is: \emph{What can we do when the minimum degree condition does not guarantee that every two vertices are reachable?}
Keevash and Mycroft \cite{KM1} studied this case. They showed that\footnote{In fact, their result is on the degree sequence on $k$-complexes, which is more general. Here we state their work in the codegree case on $k$-graphs.} for any $k$-graph $H$  with a bounded minimum codegree which is not close to the space barrier, there exists a partition $\cP$ of $V(H)$ such that a nice structure appears in the lattice generated by the edge-vectors on $\cP$.
Then they showed that $H$ contains a perfect matching if the lattice satisfies certain condition.
Their proof used the hypergraph regularity method, rather than the absorbing method.
Inspired by their work, we noticed the relation of reachability and the lattice structure and developed the lattice-based absorbing method. Roughly speaking, the reachability information provides us a partition $\cP$, for which we can find an absorbing set that works (although not in the usual sense) under the lattice structure.
This method was first used in \cite{Han14_poly} for solving a problem of Karpi\'nski, Ruci\'nski and Szyma\'nska~\cite{KRS10}, for which an asymptotic result was obtained by Keevash, Knox and Mycroft~\cite{KKM13}.
Another advantage of the method is that the lattice language allows us to use some basic knowledge in group theory, such as, subgroups and cosets (see Proposition~\ref{prop:lattice}).

\medskip
The rest of this paper is organized as follows. We show Propositions~\ref{prop:gkdll} and \ref{prop:gkd} in Section 2. In Section 3, we prove Theorem~\ref{thm:npm2} by the lattice-based absorbing method. We show Theorem~\ref{thm:npm3} in Section 4. 

\medskip
\noindent\textbf{Notations.} Throughout this paper, $x\ll y$ means that for any $y\ge 0$ there exists $x_0\ge 0$ such that for any $x\le x_0$ the following statement holds. Hierarchy of more constants are defined similarly.
For simplicity, for a set $S$ and an element $u$, we write $S\cup u$ instead of $S\cup \{u\}$.
We use boldface letters to represent vectors, for example, $\bfu, \bfv, \bfi$.

\section{Proof of Propositions~\ref{prop:gkdll} and \ref{prop:gkd}}

\begin{proof}[Proof of Proposition~\ref{prop:gkdll}]
Fix $0\le t\le d$ and consider any $d$-set $S_t$ containing exactly $t$ vertices in $V_1$. By the definition of $H_\ell^j(x)$, we have
\begin{equation}\label{eq:delta_d}
\delta_d(H_\ell^j(x)) = \min\{ \deg_{H_\ell^j(x)}(S_0), \dots, \deg_{H_\ell^j(x)}(S_d) \}.
\end{equation}
Note that the neighbors of $S_t$ are all $(k-d)$-sets with $j'$ vertices in $V_1$ such that $j'\equiv j-t$ mod $\ell+2$ and $0\le j'\le k-d$. Thus,
\begin{equation}\label{eq:St}
\deg_{H_\ell^j(x)}(S_t) = \sum_{j'\equiv j-t \bmod \ell+2, 0\le j'\le k-d} \binom{|V_1|-t} {j'} \binom{|V_2|-(d-t)}{k-d-j'}.
\end{equation}
Let $p=\min\{\ell+1, d\}$ and note that $\deg_{H_\ell^j(x)}(S_0) + \dots + \deg_{H_\ell^j(x)}(S_p) \le (1+o(1))\binom{n-d}{k-d}$.
So $\delta_d(H_{\ell}^j(x)) \le (\frac1{p+1} + o(1))\binom{n-d}{k-d}$. This implies that when $d\le \ell+1$, $g(k,d,\ell)\le 1/(d+1)$ and when $d\ge \ell+1$, $g(k,d,\ell)\le 1/{(\ell+2)}$, proving part (i).

Now consider $\ell=1$ and use $|V_1|=xn$ and $|V_2|=(1-x)n$. Then \eqref{eq:St} implies that
\begin{align*}
\deg_{H_1^j(x)}(S_t) & = \sum_{j'} \frac{x^{j'}(1-x)^{k-d-j'}n^{k-d}+o(n^{k-d})}{{j'}! (k-d-j')!} \\
&= \left( \sum_{j'} \binom{k-d}{j'}x^{j'}(1-x)^{k-d-j'} +o(1) \right)\binom{n-d}{k-d},
\end{align*}
where the sums are on all $j'\equiv j-t \bmod 3$ and $0\le j'\le k-d$. 
Define three functions $h_i(k,d, x) = \sum_{j'\equiv i \bmod 3, 0\le j'\le k-d} \binom{k-d}{j'}x^{j'}(1-x)^{k-d-j'}$ for $i=0,1,2$.
Clearly these functions do not depend on $j$ and
\begin{equation}\label{eq:gkd}
g(k,d,1) = \max_{x\in [0,1]}\min_{i=0,1,2}\{h_i(k,d,x)\}.
\end{equation}
This implies that, for any $x\in [0,1]$ and $\r>0$, there exists $n_0$ such that for any $n\ge n_0$, we have 
\begin{equation}\label{eq:n0}
\min_{i=0,1,2} \left\{ \sum_{j\equiv i \bmod 3,  0\le j\le k-d} \binom{x n} {j} \binom{(1-x)n}{k-d-j} \right\} \le (g(k,d,1)+\r/2) \binom{n-d}{k-d}.
\end{equation}

To see (ii), let $\omega_1$ be one of the nontrivial cubic roots of 1.
We consider the following polynomial $(x\omega_1 + 1- x)^{k-d}$. 
It is easy to see that we can write the polynomial in the following form
\begin{equation}\label{eq:gene}
(x\omega_1 + 1- x)^{k-d} = h_0(k,d,x) + h_1(k,d,x) \omega_1 + h_2(k,d,x) \omega_1^2.
\end{equation}
First, since $\sum_{i=0,1,2}h_i(k,d,x) = 1$, we have $g(k,d,1)\le 1/3$. 
Moreover, if $g(k,d,1) = 1/3$, then let $x_0$ be the value of $x$ that achieves this, i.e., $h_i(k,d, x_0) =1/3$ for $i=0,1,2$. Putting $x_0$ in \eqref{eq:gene} and by $1+ \omega_1 + \omega_1^2 = 0$, we get $(x_0\omega_1 + 1- x_0)^{k-d}=0$.
This implies that $x_0 = 1/(1-\omega_1)\notin \mathbb{R}$, a contradiction.

To see the lower bound, set $x=1/2$ in \eqref{eq:gene} and we get
\[
\frac{(1 + \omega_1)^{k-d}}{2^{k-d}} = h_0(k,d,1/2) + h_1(k,d,1/2) \omega_1 + h_2(k,d,1/2) \omega_1^2.
\]
For $i=0,1,2$, let $C_i = 2^{k-d} h_i(k,d,1/2)$. Note that $C_0$, $C_1$ and $C_2$ are integers such that $C_0 + C_1 + C_2 = 2^{k-d}$.
On the other hand, note that $(1 + \omega_1)^{k-d} = (-1)^{k-d}\omega_1^{2(k-d)}$, which could be $\pm 1$, $\pm \omega_1$, or $\pm \omega_1^2$, depending on the value of $k-d$.
If $(1 + \omega_1)^{k-d} =1$, then by $1+ \omega_1 + \omega_1^2 = 0$, we know that $C_0 - 1 = C_1 = C_2$. This implies that $C_0 = \lceil \frac{2^{k-d}}3 \rceil$ and $C_1 = C_2 = \lfloor \frac{2^{k-d}}3 \rfloor$.
Other cases are similar and it is easy to see that in all cases, $\min\{C_0, C_1, C_2\} = \lfloor \frac{2^{k-d}}3 \rfloor$.
Thus, by definition, we have $g(k,d,1)\ge \min_{i=0,1,2}\{h_i(k,d,1/2)\} = \lfloor \frac{2^{k-d}}3 \rfloor 2^{d-k}$.
This concludes part (ii).
\end{proof}

Here we remark that for fixed $k$ and $d$, the value of $g(k,d,1)$ can be determined simply.  For example, let $k=6$ and $d=3$ and consider $H_1^0(x)$, i.e., all 6-sets $e$ such that $|e\cap V_1|=0$, 3 or 6. By \eqref{eq:gkd}, we have
\[
g(6,3,1) = \max_{x\in [0,1]}\min\{x^3+(1-x)^3, 3x^2(1-x), 3x(1-x)^2\}. 
\]
The answer is about $0.283$ which is achieved by $x\approx 0.605$.

\begin{proof}[Proof of Proposition~\ref{prop:gkd}]
Given any $k$-set $S$ in $V(H_{\ell}^j(x))$, clearly $|S\cap V_1|\in [0,k]$.
Recall that $S\in E(H_{\ell}^j(x))$ if and only if $|S\cap V_1|\equiv j$ mod $\ell+2$. Consider the longest interval $I_j=[a,b]\subseteq [0,k]$ such that for any $i\in I_j$, $i\not\equiv j$ mod $\ell+2$. 
If there are two values of $i\in [0,k]$ such that $i\equiv j$ mod $\ell+2$,
then $|I_j| = \ell+1$. 
Otherwise, there is only one value of $i\in [0,k]$ such that $i\equiv j$ mod $\ell+2$, which cuts the whole interval $[0, k]$ into two pieces, and thus $|I_j| = \max\{j, k-j\} \ge \lceil k/2\rceil$.
Moreover, $|I_{\lfloor k/2 \rfloor}| = \min\{\ell+1, \lceil k/2\rceil\}$.
Altogether $\min_{0\le j\le \ell+1} |I_j| = \min\{\ell+1, \lceil k/2\rceil\}$.
Note that the $(a+k-b)$-sets with $a$ vertices in $V_1$ and $k-b$ vertices in $V_2$ have degree zero in $H_{\ell}^j(x)$, because all $k$-sets $S$ containing such a set satisfy that $|S\cap V_1| \in I_j$. This means $\delta_{d}(H_{\ell}^j(x)) = 0$ for $d\ge a+k-b= k+1-|I_j|$.
By definition, this implies that $g(k, d, \ell) = 0$ when
\[
d\ge \max_{0\le j\le \ell+1} k+1-|I_j| = k+1 - \min_{0\le j\le \ell+1}|I_j| = k+1 - \min\{\ell+1, \lceil k/2\rceil\} = \max\{k-\ell, \lfloor k/2 \rfloor +1\}. \qedhere
\]
\end{proof}

\section{Proof of Theorem~\ref{thm:npm2}}

For a $k$-graph $H$, we shall first identify a partition $\cP$ of $V(H)$ and then study the so-called robust edge-lattice with respect to this partition.

\subsection{A partition of the vertex set}
We start with some definitions.
We say that two vertices $u$ and $v$ are \emph{$(\beta, i)$-reachable in $H$} if there are at least $\beta n^{i k-1}$ $(i k-1)$-sets $S\subseteq V(H)$ such that both $H[S\cup u]$ and $H[S\cup v]$ have perfect matchings. We say a vertex set $U$ is \emph{$(\beta, i)$-closed in $H$} if any two vertices $u,v\in U$ are $(\beta, i)$-reachable in $H$.
For any $v\in V(H)$, let $\tilde{N}_{\beta, i}(v)$ be the set of vertices in $V(H)$ that are $(\beta, i)$-reachable to $v$.

We first show a lower bound on $|\tilde{N}_{\beta, 1}(v)|$ for all but a small fraction of vertices $v\in V(H)$. Note that we cannot guarantee that the conclusion of Lemma~\ref{lem:He} holds for all vertices in $V(H)$.
Similar proof tricks are used in \cite{HZZ_tiling}.

\begin{lemma}\label{lem:He}
Given $\delta, \e>0$ such that $\delta \ge 3\e$, there exists $\a>0$ such that the following holds for sufficiently large $n$. Let $H=(V, E)$ be an $n$-vertex $k$-graph with $\delta_1(H) \ge \delta\binom{n-1}{k-1}$, 
there exists a set $V_0'\subseteq V$ of size at most $k\e n$ such that for any $v\in V\setminus V_0'$, $|\tilde{N}_{\a, 1}(v)| \ge \frac34 \e^2 n$.
\end{lemma}

\begin{proof}
If an edge $e\in E$ contains a $(k-1)$-set $S\in \binom e{k-1}$ with $\deg_H(S)\leq \e^2 n$, then it is called \emph{weak}, otherwise called \emph{strong}. 
Note that by the definition, the number of weak edges in $H$ is at most $\binom n{k-1}\e^2n$. 
Let $H'$ be the subhypergraph of $H$ induced on strong edges.
Let
\begin{align*}
V_0'=\left\{v\in V: v \text{ is contained in at least } \e \binom{n}{k-1} \text{ weak edges}\right\}.
\end{align*}
So $|V_0'|\le k\e n$ -- otherwise there are more than $k\e n\cdot \e \binom{n}{k-1}/k=\binom n{k-1} \e^2n$ weak edges in $H$, a contradiction.

Let $\a =  \frac{\delta \e^2}{8 k!}$.
For any $x\in V\setminus V_0'$, since $x$ is contained in at most $\e \binom{n}{k-1}$ weak edges, we have
\begin{equation}\label{eq:degx}
\deg_{H'}(x)\ge \deg_H(x) - \e\binom n{k-1} \ge \delta_1(H) - \e\binom n{k-1} \ge \frac{\delta}2 \binom{n-1}{k-1}.
\end{equation}
To see $|\tilde N_{\a, 1}(x)|\ge \frac34 \e^2 n$, let
\[
D=\left\{v\in V: |N_{H'}(v)\cap N_{H'}(x)|\geq \frac{\delta \e^2}{8} \binom {n-1}{k-1}\right\}.
\]
By definition, two vertices $x, v \in V$ are $(\a, 1)$-reachable in $H'$ (so in $H$) if $|N_{H'}(v)\cap N_{H'}(x)| \ge \delta \e^2 \binom {n-1}{k-1}/8 > \a n^{k-1}$. Therefore $D\subseteq \tilde N_{\a, 1}(x)$. Let $t$ be the number of pairs $(S, u)$ such that $S\in N_{H'}(x)$ and $u\in N_{H}(S)$. Since all edges of $H'$ are strong, we have $t\geq \deg_{H'}(x) \cdot \e^2 n$. By counting, we have
\[
\deg_{H'}(x) \, \e^2 n\leq t \leq n\cdot \frac{\delta \e^2}{8}\binom{n-1}{k-1}+|D|\cdot \deg_{H'}(x),
\]
which implies $|D|\ge \e^2 n  -  \delta \e^2 n \binom{n-1}{k-1}/ (8 \deg_{H'}(x))$. 
By \eqref{eq:degx}, $|\tilde N_{\a, 1}(x)|\ge |D|\geq \frac34 \e^2n$ and we are done.
\end{proof}

We will use the following simple result from \cite{LM1} here and in the next subsection.

\begin{proposition}\cite[Proposition 2.1]{LM1}\label{prop21}
For $\e, \beta>0$ and integer $i\ge 1$, there exists $\beta_0>0$ and an integer $n_0$ satisfying the following. Suppose $H$ is a $k$-graph of order $n\ge n_0$ and there exists a vertex $x\in V(H)$ with $|\tilde{N}_{\beta, i}(x)|\ge \e^2 n/2$. Then for all $0<\beta'\le \beta_0$, $\tilde{N}_{\beta,i}(x)\subseteq \tilde{N}_{\beta',i+1}(x)$.
\end{proposition}

The main tool to identify the partition of $V(H)$ is the following lemma, which is a variant of \cite[Lemma 3.8]{Han14_poly} (see also \cite[Lemma 3.8]{HZZ_tiling}). 
Since there may be some vertices $v$ such that $|\tilde{N}_{\beta, 1}(v)|$ is small, we cannot get a perfect partition as in \cite[Lemma 3.8]{Han14_poly} -- $\cP$ will contain a \emph{trash} set $V_0$.
%Since the main part of the proof is the same as the one of \cite[Lemma 3.8]{Han14_poly}, we only explain the differences in the proofs.

\begin{lemma}\label{lem:P}
Given $0<\e \ll \delta$, there exists $\beta>0$ satisfying the following. Let $H=(V, E)$ be an $n$-vertex $k$-graph such that $\delta_1(H)\ge (\delta+k^2\e)\binom{n-1}{k-1}$. Then there is a partition $\cP$ of $V(H)$ into $V_0, V_1,\dots, V_r$ with $r\le \lfloor 1/\delta \rfloor$ such that $|V_0|\le \sqrt\e n$ and for any $i\in [r]$, $|V_i|\ge \e^2 n$ and $V_i$ is $(\beta, 2^{\lfloor 1/\delta \rfloor-1})$-closed in $H$.
\end{lemma}

\begin{proof}
%We first apply Lemma \ref{lem:He} on $H$ and get $V_0'$ of size at most $k\e n$ such that for any $v\in V\setminus V_0'$, we have $|\tilde{N}_{\a, 1}(v)| \ge \frac34 \e^2 n$. This enables us to apply \cite[Proposition 2.1]{LM1} as in the proof of \cite[Lemma 3.8]{Han14_poly}.
%Note that
%\[
%\delta_1(H[V\setminus V_0']) \ge \delta_1(H) - k\e n \binom{n-2}{k-2} \ge \delta \binom{n - |V_0'| -1}{k-1}.
%\]
%Then we run the remaining proof of \cite[Lemma 3.8]{Han14_poly} on $H[V\setminus V_0']$ (while the reachability and closedness are still defined in $H$) and get a partition $\{V_1,\dots, V_{r'}\}$ of $V\setminus V_0'$ such that each $V_i$ is $(\beta, 2^{\lfloor 1/\delta \rfloor-1})$-closed in $H$ for some $r'\le \lfloor 1/\delta \rfloor$ and $\beta \ll \e$. 
%The only difference is that we don't have a lower bound on $|V_i|$ for $i\in [r']$.
%Without loss of generality, we may assume that $|V_1|\ge \cdots \ge |V_{r'}|$.
%Let $r$ be the largest integer $i\in [r']$ such that $|V_i| \ge \e^2 n$. 
%Let $V_0=V\setminus (\bigcup_{1\le i\le r}V_i)$ and clearly $V_0'\subseteq V_0$. 
%Since $\e \ll \delta$, we have $r'\e \le k$ and thus $|V_0|\le |V_0'| + r' \e^2 n\le 2k\e n $.
%So we get the partition $\cP=\{V_0, V_1, \dots, V_r\}$, which satisfies all the desired properties.

We first apply Lemma~\ref{lem:He} with $\delta,\e$ and get $\a>0$.
Let $c=\lfloor 1/\delta \rfloor$ (then $(c+1)\delta -1>0$).
We choose constants satisfying the following hierarchy
\[
1/n\ll \beta = \beta_{c-1}\ll \beta_{c-2}\ll \cdots \ll \beta_1 \ll \beta_0 \ll \e, \a, (c+1)\delta-1.
\]

Throughout this proof, given $v\in V(H)$ and $0\le i\le c-1$, we write $\tilde{N}_{\beta_i, 2^i}(v)$ as $\tilde N_{i}(v)$ for short. 
%Note that for any $v\in V(H)$, $|\tilde N_0(v)|=|\tilde{N}_{\beta_0, 1}(v)|\ge |\tilde{N}_{\a, 1}(v)| \ge \delta' n$ because $\beta_0<\a$.
We also say $2^i$-reachable (or $2^i$-closed) for $(\beta_i, 2^i)$-reachable (or $(\beta_i, 2^i)$-closed).
We first apply Lemma \ref{lem:He} on $H$ and get a vertex set $V_0'$ of size at most $k\e n$ such that for any $v\in V\setminus V_0'$, we have $|\tilde{N}_{0}(v)|\ge |\tilde{N}_{\a, 1}(v)| \ge \frac34 \e^2 n$.
By Proposition \ref{prop21} and the choice of $\beta_i$'s, we may assume that $\tilde {N}_{i}(v)\subseteq \tilde {N}_{i+1}(v)$ for all $0\le i<c-1$ and all $v\in V(H')$.
Hence, if $W\subseteq V(H)$ is $2^i$-closed in $H$ for some $i\le c-1$, then $W$ is $ 2^{c-1}$-closed.
Let $n'=|V\setminus V_0|$ and $H' = H[V\setminus V_0']$.
Note that
\[
\delta_1(H') \ge \delta_1(H) - k\e n \binom{n-2}{k-2} \ge \delta \binom{n' -1}{k-1}.
\]

Recall that two vertices $u$ and $v$ are 1-reachable in $H$ if $|N_H(u)\cap N_H(v)|\ge \beta_0 n^{k-1}$. We first note that any set of $c+1$ vertices in $V(H')$ contains two vertices that are 1-reachable to each other because 
%updated 01/09/2014
$\delta_1(H')\ge \delta\binom{n'-1}{k-1}$ and $(c+1)\delta-1\ge \binom{c+1}2\beta_0$. 
Also we can assume that there are two vertices that are not $2^{c-1}$-reachable to each other, as otherwise $V(H')$ is $2^{c-1}$-closed and we get a partition $\cP_0=\{V_0', V(H')\}$.

Let $d$ be the largest integer such that
there exist $v_1,\dots, v_{d}\in V(H')$ such that no pair of them are $2^{c+1-d}$-reachable to each other. 
Note that $d$ exists by our assumption and $2\le d\le c=\lfloor 1/\delta \rfloor$ by our observation. 
Fix such $v_1,\dots, v_{d}\in V(H')$, by Proposition \ref{prop21}, we may assume that any two of them are not $2^{c-d}$-reachable to each other. Consider $\tilde N_{c-d}(v_i)$ for all $i\in [d]$. Then we have the following facts.
\begin{enumerate}[(i)]
\item Any $v\in V(H)\setminus\{v_1,\dots, v_{d}\}$ must be in $\tilde N_{c-d}(v_i)$ for some $i\in [d]$, as otherwise $v, v_1,\dots, v_{d}$ contradicts the definition of $d$. 

\item $|\tilde N_{c-d}(v_i)\cap \tilde N_{c-d}(v_j)|<\e n$ because $v_i, v_j$ are not $2^{c+1-d}$-reachable to each other.
Indeed, otherwise we get at least
\[
\e n (\beta_{c-d}n^{2^{c-d}k-1} - n^{2^{c-d}k-2}) (\beta_{c-d}n^{2^{c-d}k-1} - 2^{c-d}k n^{2^{c-d}k-2})  \ge \beta_{c+1-d} n^{2^{c+1-d}k-1}
\]
reachable $(2^{c+1-d}k-1)$-sets for $v_i, v_j$, which means that they are $2^{c+1-d}$-reachable, a contradiction. 
\end{enumerate}
%Note that (ii) and $|\tilde N_{c-d}(v_i)|\ge |\tilde{N}_{0}(v_i)| \ge \delta' n$ for $i\in [d]$ imply $d\delta' n - \binom d2 \e n\le n$. So we have $d\le (1+d^2 \e)/\delta'$. Since $\e \le \a \ll \delta'$, we have $d\le \lfloor 1/\delta' \rfloor$ and thus, $d\le \min\{\lfloor 1/\delta \rfloor, \lfloor 1/\delta' \rfloor\}$.

For $i\in [d]$, let $U_i=(\tilde N_{c-d}(v_i)\cup \{v_i\})\setminus \bigcup_{j\in [d]\setminus \{i\}} \tilde N_{c-d}(v_j)$. Note that for $i\in [d]$, $U_i$ is $2^{c-d}$-closed. Indeed, if there exist $u_1, u_2\in U_i$ that are not $2^{c-d}$-reachable to each other, then $\{u_1, u_2\}\cup (\{v_1,\dots, v_{d}\}\setminus\{v_{i}\})$ contradicts the definition of $d$.

Let $U_0=V(H)\setminus (V_0'\cup U_1\cup\cdots \cup U_{d})$. By (i) and (ii), we have $|U_0|\le \binom{d}{2}\e n$. We add vertices of $U_0$ and the vertices of $U_i$ with $|U_i|\le \e^2 n$ to $V_0'$ and denote the resulting set by $V_0$.
Let the resulting partition of $V(H)$ be $V_0, V_1,\dots, V_{r}$ for some $r\le d\le \lfloor 1/\delta \rfloor$. By definition we have $|V_i|\ge \e^2 n$ for $i\in [r]$ and $|V_0|\le |V_0'| + |U_0| + d \e^2 n \le \sqrt\e n$, as $\e \ll \delta$. Moreover, each $V_i$ is $2^{c-d}$-closed, and thus $2^{c-1}$-closed. 
\end{proof}

\subsection{The robust edge-lattice}

We need some definitions from \cite{KM1}. 
Fix an integer $r>0$, let $H$ be a $k$-graph and let $\cP=\{V_0, V_1,\dots, V_r\}$ be a partition of $V(H)$.
By Lemma~\ref{lem:P}, $V_0$ does not have the reachability information. So when we work on the edge-lattice, we consider the $r$-dimensional vectors on the parts of $\cP$ except $V_0$. 
Formally, the \emph{index vector} $\mathbf{i}_{\cP}(S)\in \mathbb{Z}^r$ of a subset $S\subset V(H)$ with respect to $\cP$ is the vector whose coordinates are the sizes of the intersections of $S$ with each part of $\cP$ except $V_0$, i.e., $\mathbf{i}_{\cP}(S)_{V_i}=|S\cap V_i|$ for $i\in [r]$.
We call a vector $\mathbf{i}\in \mathbb{Z}^r$ an \emph{$s$-vector} if all its coordinates are nonnegative and their sum equals $s$ and denote the set of all $s$-vectors by $I_s^r$.
Given $\mu>0$, a $k$-vector $\mathbf{v}$ is called a $\mu$\emph{-robust edge-vector} if at least $\mu |V(H)|^k$ edges $e\in E(H)$ satisfy $\mathbf{i}_\cP(e)=\mathbf{v}$.
Let $I_{\cP}^{\mu}(H)\subseteq I_k^r$ be the set of all $\mu$-robust edge-vectors and let $L_{\cP}^{\mu}(H)$ be the lattice (additive subgroup) generated by the vectors of $I_{\cP}^{\mu}(H)$.
For $j\in [r]$, let $\mathbf{u}_j\in \mathbb{Z}^r$ be the $j$-th \emph{unit vector}, namely, $\mathbf{u}_j$ has 1 on the $j$-th coordinate and 0 on other coordinates.
A \emph{transferral} is the vector $\bfu_i - \bfu_j$ for some $i\neq j$.

Suppose $I$ is a set of $k$-vectors in $\mathbb{Z}^r$ and $J$ is a set of vector in $\mathbb{Z}^r$ such that any $\bfi\in J$ can be written as a linear combination of vectors in $I$, namely,
\begin{equation}\label{eq:Cmax}
\bfi=\sum_{\bfv\in I}a_{\bfv}\bfv. 
\end{equation}
We denote by $C(r, k, I, J)$ as the maximum of $|a_{\bfv}|, \bfv\in I$ over all $\bfi\in J$ and $C(k',k,J) := \max_{r\le k', I\subseteq I_k^r} C(r,k,I,J)$ for some integer $k'$.

Given a $k$-graph $H$ with $\delta_1(H)\ge (\delta + k^2\e) \binom{n-1}{k-1}$ and let $\cP=\{V_0, V_1, \dots, V_r\}$ be the partition of $V(H)$ output by Lemma~\ref{lem:P}. We pick a constant $0<\mu\ll \e$ and consider $I_{\cP}^{\mu}(H)$ and $L_{\cP}^{\mu}(H)$.
Our next result shows that $V_i\cup V_j$ is closed in $H$ if $\bfu_i - \bfu_j\in L_{\cP}^{\mu}(H)$, which means that, we can merge $V_i$ and $V_j$ and keep the closedness.

\begin{lemma}\label{clm:trans}
Let $0< \{\mu, \beta\}\ll \e\ll 1/i_0, 1/k', 1/k$, then there exist $0<\beta'\ll \{\mu, \beta\}$ and an integer $t \ge i_0$ such that the following holds.
Let $n$ be a sufficiently large integer.
Suppose $H$ is an $n$-vertex $k$-graph and $\cP=\{V_0, V_1, \dots, V_r\}$ is a partition with $r\le k'$ such that $|V_0|\le \sqrt\e n$ and for any $i\in [r]$, $|V_i|\ge \e^2 n$ and $V_i$ is $(\beta, i_0)$-closed in $H$.
If $\bfu_i - \bfu_j\in L_{\cP}^{\mu}(H)$, then $V_i\cup V_j$ is $(\beta', t)$-closed in $H$.
\end{lemma}

\begin{proof}
We prove the claim for $i=1$ and $j=2$. 
Note that it suffices to show that any $x_1\in V_1$ and $x_2\in V_2$ are $(\beta'', t)$-reachable for some $\beta''>0$ and integer $t\ge i_0$.
Indeed, since $V_1$ and $V_2$ are $(\beta, i_0)$-closed in $H$, by Proposition~\ref{prop21}, there exists $\beta'''$ such that $V_1$ and $V_2$ are $(\beta''', t)$-closed in $H$ and by the assumption above, we get that $V_1\cup V_2$ are $(\beta', t)$-closed in $H$, where $\beta' = \min\{\beta'', \beta'''\}$.

Let $J=\{\bfu_1 - \bfu_2\}$.
By our assumption, there are nonnegative integers $p_{\mathbf{v}}, q_{\mathbf{v}}\le C(k',k,J)$, for each $\mathbf{v}\in I_{\cP}^{\mu}(H)$, such that
\begin{align}
\bfu_1 - \bfu_2=\sum_{\mathbf{v}\in I_{\cP}^{\mu}(H)} {(p_{\mathbf{v}} - q_{\mathbf{v}})\mathbf{v}}
\quad i.e., \quad
\sum_{\mathbf{v}\in I_{\cP}^{\mu}(H)}{q_{\mathbf{v}}}\mathbf{v} + \mathbf{u}_1=\sum_{\mathbf{v}\in I_{\cP}^{\mu}(H)}{p_{\mathbf{v}}}\mathbf{v} + \bfu_2. \label{regroup11}
\end{align}
By comparing the sums of all the coordinates from two sides of either equation in \eqref{regroup11}, we obtain that
\[
\sum_{\mathbf{v}\in I_{\cP}^{\mu}(H)}{p_{\mathbf{v}}}=\sum_{\mathbf{v}\in I_{\cP}^{\mu}(H)}{q_{\mathbf{v}}}.
\]
Denote this constant by $C'$ and note that $C'\le C(k',k,J) |I_{\cP}^{\mu}(H)| \le C(k',k,J) \binom{k+r-1}{r-1}$, which is independent of $n$.
Since $n$ is large enough, we have $C' \ll n$.

Fix $x_1\in V_1$ and $x_2\in V_2$. 
For each $\mathbf{v}\in I_{\cP}^{\mu}(H)$, we select $p_{\mathbf{v}} + q_{\mathbf{v}}$ disjoint edges with index vector $\mathbf{v}$ that do not contain $x_1$ or $x_2$, and form two disjoint matchings $M^p$ and $M^q$, where $M^p$ consists of $p_{\mathbf{v}}$ edges with index vector $\bfv$ for all $\bfv\in I_{\cP}^{\mu}(H)$, and $M^q$ consists of $q_{\mathbf{v}}$ edges with index vector $\mathbf{v}$ for all $\bfv\in I_{\cP}^{\mu}(H)$.
Note that $|V(M^p)|=|V(M^q)|=kC'$.
When we select any edge, we need to avoid at most $2k C'$ vertices, which are incident to at most $2k C' n^{k-1}\leq {\mu} n^k/2$ edges, as $n$ is large enough. Therefore, the number of choices for the two matchings is at least $({\mu} n^k/2)^{2C'}$.

By \eqref{regroup11}, we have $\mathbf{i}_{\cP}(V(M^q))+\bfu_1=\mathbf{i}_{\cP}(V(M^p)) + \bfu_2$.
Fix two vertices $x_1'\in V(M^p)\cap V_1$ and $x_2'\in V(M^q)\cap V_2$.  We match each of the vertices from $V(M^p)\setminus \{x_1'\}$ with a different vertex from $V(M^q)\setminus \{x_2'\}$ such that two matched vertices are from the same part of $\cP$ and thus are $(\beta, i_0)$-reachable to each other. We next select a reachable $(i_0 k-1)$-set for each pair of the matched vertices such that all these $kC'-1$ $(i_0 k - 1)$-sets are vertex disjoint and also disjoint from $V(M^p\cup M^q)\cup\{x_1, x_2\}$. There are at least $\frac{\beta}2 n^{i_0 k-1}$ choices for each of these $(i_0 k-1)$-sets. Finally,  since $x_1$ and $x_1'$ and respectively, $x_2$ and $x_2'$ are $(\beta, i_0)$-reachable, we pick two vertex-disjoint reachable $(i_0 k -1)$-sets for them such that these two $(i_0 k -1)$-sets are also disjoint from all existing $(i_0 k -1)$-sets and $V(M^p\cup M^q)\cup\{x_1, x_2\}$. The union of these $kC'+1$ $(i_0 k -1)$-sets and $V(M^p\cup M^q)$ forms a reachable $(i_0 k^2 C'+kC'+i_0 k-1)$-set for $x_1$ and $x_2$. There are at least
\[
\left(\frac{\mu}2n^k\right)^{2C'}\left( \frac{\beta}2 n^{i_0k-1}\right)^{kC'+1}= \left(\frac{\mu}2\right)^{2C'}\left(\frac{\beta}2\right)^{kC'+1} n^{i_0 k^2 C'+kC'+i_0k-1}
\]
such reachable sets. Thus, we take $\beta'' = (\frac{\mu}2)^{2C'} (\frac{\beta}2)^{kC'+1}$ and $t = i_0 k C' +C'+i_0$ and the proof is complete.
\end{proof}

\subsection{Proof of Theorem~\ref{thm:npm2}}

Fix an integer $i>0$. For a $k$-vertex set $S$, we say a set $T$ is an \emph{absorbing $i$-set for $S$} if $|T|=i$ and both $H[T]$ and $H[T\cup S]$ contain perfect matchings. 
We use the absorbing lemma from \cite[Lemma 3.4]{Han14_poly} with some quantitative changes, but it easily follows from the original proof.

\begin{lemma}\cite{Han14_poly}\label{lem:abs}
%Let $r\le k$ and $C:=C(I_{2k}^r)$.
Suppose $r\le k$ and
\[
1/n \ll \a \ll \{\beta, \mu\} \ll \{1/k, 1/t\}.
\]
Suppose that $\cP_0=\{V_0, V_1, \dots, V_r\}$ is a partition of $V(H)$ such that for $i\in [r]$, $V_i$ is $(\beta, t)$-closed.
Then there is a family $\F_{abs}$ of disjoint $tk^2$-sets with size at most $\beta n$ such that $H[V(\F_{abs})]$ contains a perfect matching and every $k$-vertex set $S$ with $\bfi_{\cP_0}(S)\in I_{\cP_0}^{\mu} (H)$ has at least $\a n$ absorbing $t k^2$-sets in $\F_{abs}$. 
\end{lemma}

Another key step in the proof of Theorem~\ref{thm:npm2} is the following proposition. We postpone its proof to the next subsection.

\begin{proposition}\label{prop:lattice}
Given $\min\{3, k/2\} \le d\le k-2$ or $(k,d) = (5,2)$, $0<\mu \ll \e \ll \r$ and let $n$ be sufficiently large.
Let $H$ be an $n$-vertex $k$-graph with $\delta_d(H) \ge (g(k,d,1)+\r)\binom{n-d}{k-d}$ and let $\cP = \{V_0, \dots, V_r\}$ be a partition of $V(H)$ with $r\le 3$ such that $|V_0|\le \sqrt\e n$ and for any $i\in [r]$, $|V_i|\ge \e^2 n$, and $L_{\cP}^{\mu}(H)$ contains no transferral. Then for any $U\subseteq V(H)\setminus V_0$ with $|U|=k+1$, there exists $i\in [r]$ such that $\bfi_{\cP}(U) - \bfu_i\in L_{\cP}^{\mu}(H)$.
\end{proposition}

\medskip
\begin{proof}[Proof of Theorem~\ref{thm:npm2}]
The lower bound in the theorem is shown by Construction~\ref{con:db1} and the definition of $c^*_d(k)$.
For the upper bound, note that Proposition~\ref{prop:lattice} does not cover the cases when $d=1$ and $k\ge 4$ or $d=2$ and $k\ge 6$. In fact, in all of these cases, we have that $c^*_d(k) \ge 1-(1-1/k)^{k-d} > 1/2\ge g(k,d,1)$,\footnote{This can be proved by showing that $f(k) = 1-(1-1/k)^{k-d}$ is increasing for $d=1,2$.} and thus Theorem~\ref{thm:npm2} follows from Corollary~\ref{cor:HPS}.
So we may assume that $k$ and $d$ are integers such that $\min\{3, k/2\} \le d\le k-2$ or $(k,d) = (5,2)$.

Fix $\r>0$ and pick $0< \mu_0 \ll \e_0 \ll \r$.
We first apply Lemma~\ref{lem:P} with $\delta = g(k,d,1)+\r/2$, $\e_0$ and get $0<\beta_0\ll \e_0$.
We then apply Lemma~\ref{clm:trans} with $i_0=4$, $k'=k$ and $\{\beta_0, \mu_0\} \ll \e_0 \ll 1/i_0$ and get $0<\beta_1 \ll \{\beta_0, \mu_0\}$ and $t_0 \ge 4$. 
Pick $\beta_1', \mu_1, \e'>0$ such that $\{\beta_1', \mu_1\}\ll \e'\ll 1/t_0$ and $\beta_1'\le \beta_1$.
We apply Lemma~\ref{clm:trans} again with $i_0=t_0$, $k'=k$ and $\{\beta_1', \mu_1\}\ll \e'\ll 1/t_0$ and get $t\ge t_0$ and $0<\beta_2 \ll \{\beta_1', \mu_1\}$. Let $C:=C(k',k,I_{2k}^r)$.
At last, we pick a new round of constants
\[
1/n_0 \ll \a \ll \beta' \ll \{\beta, \mu\} \ll \e \ll \r, 1/k, 1/t, 1/C,
\]
such that we can apply the auxiliary results above.

For any $n\ge n_0$ such that $n \bmod k = \ell \ge 1$, let $H=(V, E)$ be a $k$-graph with $\delta_d(H) \ge (\max\{c^*_d(k), g(k,d,1)\}+\r)\binom{n-d}{k-d}$. 
Note that $\delta_1(H) \ge (g(k,d,1)+\r)\binom{n-1}{k-1}$.
We first apply Lemma~\ref{lem:P} on $H$ with $\delta = g(k,d,1)+\r/2$. Since $g(k,d,1)\ge 1/4$, we get a partition $\cP = \{V_0, V_1', \dots, V_r'\}$ 
with $r\le 3$ such that $|V_0|\le \sqrt\e n$ and for any $i\in [r]$, $|V_i'|\ge \e^2 n$ and $V_i'$ is $(\beta, 4)$-closed in $H$.
If $\bfu_i - \bfu_j\in L_{\cP}^{\mu}(H)$ for some $i, j\in [r]$, $i\neq j$, then we merge $V_i$ and $V_j$ to one part and by Lemma~\ref{clm:trans}, $V_i\cup V_j$ is $(\beta'', t')$-closed for some $\beta''>0$ and $t'\ge 4$. We greedily merge the parts until there is no transferral in the $\mu$-robust edge-lattice.
Let $\cP_0 = \{V_0, \dots, V_{r'}\}$ be the resulting partition for some $1\le r'\le 3$.
Note that we have applied Lemma~\ref{clm:trans} at most twice and by Proposition~\ref{prop21}, we conclude that for each $i\in [r']$, $V_i$ is $(\beta', t)$-closed by the choice of $\beta'$.
We apply Lemma~\ref{lem:abs} on $H$ and get $\F_{abs}$ such that $|V(\F_{abs})|\le tk^2 \beta' n$.

We build a matching $M_1$ in $H$ as follows. 
First note that $|I_{\cP_0}^{\mu}(H)|\le \binom{k + r' - 1}{r'-1}\le \binom{k+2}{2}$.
For each $\bfv\in I_{\cP_0}^{\mu}(H)$, we greedily pick a matching $M_{\bfv}$ of size $C\a^2 n$ such that $\bfi_{\cP_0}(e) = \bfv$ for every $e\in M_{\bfv}$. Then let $M_1$ be the union of $M_{\bfv}$ for all $\bfv\in I_{\cP_0}^{\mu}(H)$.
It is possible to pick $M_1$ because there are at least $\mu n^k$ edges $e$ with $\bfi_{\cP_0}(e) = \bfv\in I_{\cP_0}^{\mu}(H)$.
Indeed, since $\a \ll\beta' \ll \mu\ll 1/t, 1/C$, we have
\[
|V(M_1)\cup V(\F_{abs})|\le k |I_{\cP_0}^{\mu}(H)| C\a^2 n + tk^2 \beta' n < \mu n,
\]
which implies that the number of edges intersecting these vertices is less than $\mu n^k$ and we are done.

Next we will greedily match the vertices in $V_0\setminus V(\F_{abs})$. 
More precisely, we will find a matching $M_2$ that covers all vertices of $V_0\setminus V(\F_{abs})$. 
Note that $|M_2|\le |V_0|\le \sqrt\e n$.
When we greedily match a vertex $v\in V_0\setminus V(\F_{abs})$, we need to avoid at most
$k|M_2| + |V(M_1)\cup V(\F_{abs})| \le k\sqrt\e n + \mu n\le 2k\sqrt \e n$ vertices, and thus at most $2k\sqrt \e n^{k-1}$ $(k-1)$-sets. Since $\delta_1(H)> \r \binom{n-1}{k-1} > 2k\sqrt\e n^{k-1}$, we can always find a desired edge containing $v$ and add it to $M_2$.

Let $V' = V\setminus (V(\F_{abs})\cup V(M_1\cup M_2))$ and $H' = H[V']$. By previous calculations, we have $|V(\F_{abs})\cup V(M_1\cup M_2)|\le 2 k\sqrt\e n$ and thus
\[
\delta_d(H') \ge (c^*_d(k) + \r)\binom{n-d}{k-d} - 2 k\sqrt \e n\cdot n^{k-d-1} \ge (c^*_d(k) + \r/2)\binom{|V'|-d}{k-d},
\]
as $\e \ll \r$.
So by the definition of $c^*_d(k)$ and that $|V'|\ge n/2$ is large enough, we can find a matching $M_3$ in $H'$ which leaves at most $\a^2 |V'|\le \a^2 n$ vertices uncovered.

Now we absorb the uncovered vertices by $\F_{abs}$. Fix any set $U$ of $k+1$ uncovered vertices, by Proposition~\ref{prop:lattice}, there exists $i\in [r']$ such that $\bfi_{\cP_0}(U) - \bfu_i\in L_{\cP_0}^{\mu}(H)$.
Note that this does not guarantee that we can delete one vertex $v$ from $U$ such that $\bfi_{\cP_0}(U\setminus \{v\}) \in L_{\cP_0}^{\mu}(H)$, because it is possible that $U\cap V_i=\emptyset$ for the $i$ returned by the proposition.
By the degree condition, there is a vector $\bfv \in I_{\cP_0}^{\mu}(H)$ such that $\bfv_{V_i}>0$ and note that $M_1$ contains $C\a^2 n$ edges with index vector $\bfv$. Fix one such edge $e$ and a vertex $v\in e\cap V_i$.
We delete $e$ from $M_1$ and let $U' = U\cup (e\setminus \{v\})$. 
Clearly, $\bfi_{\cP_0}(U') \in L_{\cP_0}^{\mu}(H)$ and $|U'|= 2k$.
Thus, by definition, there exist nonnegative integers $b_\bfv, c_\bfv$ for all $\bfv\in I_{\cP_0}^{\mu}(H)$ such that
\[
\bfi_{\cP_0}(U') = \sum_{\bfv\in I_{\cP_0}^{\mu}(H)}b_{\bfv}\bfv - \sum_{\bfv\in I_{\cP_0}^{\mu}(H)}c_{\bfv}\bfv \quad i.e., \quad \bfi_{\cP_0}(U') + \sum_{\bfv\in I_{\cP_0}^{\mu}(H)}c_{\bfv}\bfv= \sum_{\bfv\in I_{\cP_0}^{\mu}(H)}b_{\bfv}\bfv.
\]
By the definition of $C$, we know that $b_\bfv, c_\bfv\le C$.
For each $\bfv\in I_{\cP_0}^{\mu}(H)$, we pick $c_\bfv$ edges in $M_1$ with index vector $\bfv$.
By the equation above, the union of these edges and $U'$ can be partitioned as a collection of $k$-sets, which contains exactly $b_\bfv$ $k$-sets $F$ with $\bfi_{\cP_0}(F) = \bfv$ for each $\bfv\in I_{\cP_0}^{\mu}(H)$.
We repeat the process at most $\a^2 n/k$ times until there are exactly $\ell$ vertices left.
Note that for each $\bfv\in I_{\cP_0}^{\mu}(H)$, our algorithm consumes at most $(1+C)\a^2 n/k < C\a^2 n$ edges from $M_1$ with index vector $\bfv$ -- this is possible by the definition of $M_1$.
Moreover, after the process, we get at most $(2+ |I_{\cP_0}^{\mu}(H)| C)\a^2 n/k\le (2+\binom{k+2}{2}C)\a^2 n/k < \a n$ $k$-sets $S$ with $\bfi_{\cP_0}(S) \in I_{\cP_0}^{\mu}(H)$, because $\a \ll 1/k, 1/C$.
By the definition of $\F_{abs}$, we can greedily absorb them by $\F_{abs}$ and get a matching $M_4$.
Thus, we get a near perfect matching of $H$.
\end{proof}

\subsection{The transferral-free lattices}

In this subsection we prove Proposition~\ref{prop:lattice}. We study the lattice structure $L_{\cP}^{\mu}(H)$ when it contains no transferral.

%For $t\ge 2$, let $I_t^r$ be the set of $t$-vectors on a partition with $r$ parts.
Fix $1\le p\le k-1$ and any $p$-vector $\bfv$, its \emph{neighborhood $N(\bfv)$} is the set of vectors $\bfv'$ such that $\bfv + \bfv'\in L_{\cP}^{\mu}(H)$.
Note that by definition, the vectors in $N(\bfv)$ may contain negative coordinates.
Moreover, assume $r=2$, we claim that $N(\bfv)\cap I_{k-p}^2\neq \emptyset$ for any $1\le p\le d$ and any $p$-vector $\bfv = (i, p-i)$. 
Indeed, otherwise, let $\bfv$ be a $p$-vector such that $N(\bfv)\cap I_{k-p}^2= \emptyset$.
This implies that the number of edges in $H[V\setminus V_0]$ with index vector $\bfi$ such that $\bfi - \bfv \in I_{k-p}^2$ is at most $|I_{k-p}^2| \mu n^k \le 2^{k-p} \mu n^k$.
Let $A_{\bfv}$ be the set of all $p$-sets $S$ with $\bfi_{\cP}(S) = \bfv$ and thus $|A_\bfv| = \binom{|V_1|}{i}\binom{|V_2|}{p-i} \ge \binom{\e^2 n}p$. 
By averaging, there is a $p$-set $S$ in $A_\bfv$ such that
\[
\deg_H(S) \le |I_{k-p}^2|\mu n^k / |A_\bfv| + |V_0| n^{k-p-1} \le 2^{k-p}\mu n^k / \binom{\e^2 n}{p} + \sqrt\e n^{k-p} < \r\binom{n-p}{k-p},
\]
by $\mu \ll \e \ll \r$. Since $p\le d$, this contradicts that $\delta_d(H)\ge (g(k,d,1)+\r)\binom{n-d}{k-d}$.
Note that a similar argument works for $r=3$, namely, for any $p$-vector $\bfv = (i, i', p-i-i')$ with $1\le p\le d$, $N(\bfv)\cap I_{k-p}^3\neq \emptyset$.

\begin{claim}\label{clm:vec_lat}
Given $\min\{3, k/2\} \le d\le k-2$ or $(k,d) = (5,2)$, $0<\mu \ll \e \ll \r$ and let $n$ be sufficiently large.
Let $H$ and $\cP$ be as defined in Proposition~\ref{prop:lattice}.
If $r=2$, then $(2,-2)\in L_{\cP}^{\mu}(H)$. If $r=3$, then $(-2,1,1), (1,-2,1), (1,1,-2)\in L_{\cP}^{\mu}(H)$.
\end{claim}

\begin{proof}
First assume that $r=2$. Fix $(a_0, b_0)\in I_{\cP}^{\mu}(H)$.
For the sake of a contradiction, assume that $(2,-2)\notin L_{\cP}^{\mu}(H)$. Let $L_0$ be the sublattice (subgroup) of $L_{\cP}^{\mu}(H)$ such that $(a,b)\in L_0$ if $a+b=0$ and let $L_{k-d} = \{(a,b)\mid a+b = k-d\}$. 
Let $t$ be the smallest positive integer such that $(t, -t)\in L_0$ and it is easy to see that $L_0$ is generated by $(t, -t)$. By our assumption, $(1, -1), (2,-2)\notin L_{\cP}^{\mu}(H)$, and thus $t\ge 3$.
Let $t_0 = \min\{t, k-d+1\}$. It is easy to see that $L_0$ partitions $L_{k-d}$ into $t_0$ cosets $C_0,\dots, C_{t_0-1}$ such that $C_i = (k-d-i, i) + L_0$ for all $0\le i\le t_0 - 1$.
For any $0\le j\le d$ and $\bfv_j := (d-j, j)$, we have
\[
N(\bfv_j) = (a_0, b_0) - (d-j, j) + L_0 = (k-d-(b_0-j), b_0-j) + L_0.
\]
This means that $N(\bfv_j)= C_{i_j}$, where $i_j\equiv b_0-j$ mod $t_0$. 
We split into two cases.

\medskip
\noindent\textbf{Case 1.} $t_0= 3$.
Note that if $N(\bfv_0)\cap N(\bfv_1)\neq \emptyset$, say, $\bfi\in N(\bfv_0)\cap N(\bfv_1)$, then we have $\bfi+\bfv_0, \bfi+\bfv_1\in L_{\cP}^{\mu}(H)$ and thus $(2,-2)=2(\bfv_0 - \bfv_1)\in L_{\cP}^{\mu}(H)$, a contradiction.
Similarly $N(\bfv_1)\cap N(\bfv_2) = N(\bfv_0)\cap N(\bfv_2)=\emptyset$ and thus $N(\bfv_0)$, $N(\bfv_1)$ and $N(\bfv_2)$ are pairwise disjoint.
%Since $(1,-1), (2,-2)\notin L_{\cP}^{\mu}(H)$, we know that $N(\bfv_0)$, $N(\bfv_1)$ and $N(\bfv_2)$ are pairwise disjoint, and
Consequently $\{N(\bfv_0), N(\bfv_1), N(\bfv_2)\} = \{C_0, C_1, C_2\}$.
Recall that $C_i = (k-d - i, i)+ L_0$ for $i=0,1,2$, namely, all $(k-d)$-vectors with their second coordinates congruent to $i$ modulo 3. 
For $i=0,1,2$, let $B_i$ be the collection of $d$-sets with exactly $j$ vertices in $V_2$ such that $j\equiv i$ mod 3 and thus $|B_i| \ge \binom{|V_1|}{d-i}\binom{|V_2|}{i}\ge \binom{\e^2 n}{d}$.
By averaging, there exists a $d$-set $S_i$ in $B_i$ such that 
\begin{align*}
\deg_H(S_i) &\le \sum_{j\equiv i \bmod 3, 0\le j\le k-d} \binom{|V_2|} {j} \binom{|V_1|}{k-d-j} + 2^{k-d}\mu n^k/ |B_i| + |V_0| n^{k-d-1} \\
&< \sum_{j\equiv i \bmod 3, 0\le j\le k-d} \binom{|V_2|} {j} \binom{|V_1|}{k-d-j} + \frac{\r}2\binom{n-d}{k-d},
\end{align*}
by $\mu \ll \e \ll \r$. Thus we have
\[
\min_{i=0,1,2} \deg_H(S_i) < \min_{i=0,1,2} \left\{ \sum_{j\equiv i \bmod 3,  0\le j\le k-d} \binom{|V_2|} {j} \binom{|V_1|}{k-d-j} \right\} + \frac{\r}2\binom{n-d}{k-d}.
\]
By \eqref{eq:n0}, we know that the first term in the right hand side of the inequality above is at most $(g(k,d,1)+\r/2)\binom{n-d}{k-d}$, as $n$ is large enough.
Thus, we get $\delta_d(H)\le \min_{i=0,1,2} \deg_H(S_i) < (g(k,d,1)+\r)\binom{n-d}{k-d}$, a contradiction.

\medskip
\noindent\textbf{Case 2.} $t_0 \ge 4$. Since $t_0 = \min\{t, k-d+1\}$, we have $t \ge 4$ and $d\le k-3$. 
First assume that $d\ge 3$.
Recall that $N(\bfv_j)= C_{i_j}$, where $i_j\equiv b_0-j$ mod $t_0$. Since $t_0 \ge 4$, $b_0 -j$ for $j\in\{0,\dots, 3\}$ are pairwise distinct. 
This implies that $N(\bfv_0),\dots, N(\bfv_3)$ are four distinct classes.
For $j=0,\dots, 3$, consider the following sums
\[
\sum_{(k-d-i_j, i_j)\in C_{i_j}, 0\le i_j\le k-d} \binom{|V_1|}{k-d-i_j} \binom{|V_2|} {i_j} 
\]
and note that their sum is at most $\binom{n - |V_0|}{k-d}$. By the pigeonhole principle, there exists $j'$ such that the $j'$-th sum is at most $\frac1{4}\binom{n - |V_0|}{k-d}$. This implies that
\[
\delta_d(H) \le \frac1{4}\binom{n - |V_0|}{k-d} + |I_{k-d}^2|\mu n^k/\binom{\e^2 n}{d} + |V_0| n^{k-d-1} < (1/4 + \r)\binom{n - d}{k-d},
\]
by $\mu \ll \e \ll \r$. This is a contradiction because $g(k,d,1) \ge \lfloor \frac{2^{k-d}}3 \rfloor 2^{d-k} \ge 1/4$.

Now we assume $d=2$ and by $d\le k-3$, we have $k=5$. Thus, $k-d=3$, $t_0=4$ and for $i=0,\dots, 3$, $C_i\cap I_3^2 = \{(3-i, i)\}$.
Since $(1,-1), (2,-2)\notin L_{\cP}^{\mu}(H)$, we know that for $i=0,1,2$, $N(\bfv_j)\cap I_3^2 = C_{i_j}\cap I_3^2  = \{(3-i_j, i_j)\}$ are distinct. 
So $\{ N(\bfv_{i})\cap I_3^2 \}_{i=0,1,2}$ equals $I_3^2\setminus \{(3,0)\}$, $I_3^2\setminus \{(2,1)\}$, $I_3^2\setminus \{(1,2)\}$ or $I_3^2\setminus \{(0,3)\}$.
If $\{ N(\bfv_{i})\cap I_3^2 \}_{i=0,1,2} = I_3^2\setminus \{(0,3)\} = \{ \{(3, 0)\}, \{(2,1)\}, \{(1,2)\}\}$, then
\begin{align*}
&\min\left\{\binom{|V_1|}{3} , \binom{|V_1|}{2} {|V_2|}, {|V_1|} \binom{|V_2|} {2}\right\} \\
\le& \min\left\{\binom{|V_1|}{3}+\binom{|V_2|}{3} , \binom{|V_1|}{2} {|V_2|}, {|V_1|} \binom{|V_2|} {2}\right\} < (g(5,2,1) + \r/2) \binom{n-2}{3},
\end{align*}
by \eqref{eq:n0} and that $n$ is large enough.
By averaging, we get that
\[
\delta_2(H) < (g(5,2,1) + \r/2) \binom{n-2}{3} + 2^{3}\mu n^5/\binom{\e^2 n}{2} + |V_0| n^{2} < (g(5,2,1) + \r)\binom{n - 2}{3},
\]
by $\mu \ll \e \ll \r$, a contradiction. The other three cases are similar.

\medskip
Second we assume that $r=3$.
Indeed, in this case, it suffices to have $\delta_2(H)\ge (1/4 +\r)\binom{n-2}{k-2}$. 
Consider the set of 2-vectors
\[
I_2^3 = \{(2,0,0), (0,2,0), (0,0,2), (1,1,0), (1,0,1), (0,1,1)\}.
\]
Note that $N((1,1,0))$, $N((1,0,1))$ and $N((0,1,1))$ are pairwise disjoint -- because $L_{\cP}^{\mu}(H)$ contains no transferral. Similarly, $N((2,0,0))\cap N((1,1,0)) = \emptyset$ and $N((2,0,0))\cap N((1,0,1)) = \emptyset$ (and similar equations hold for other vectors).
Moreover, recall that for all $\bfv\in I_2^3$, $N(\bfv)\cap I_{k-2}^3$ is nonempty.
Thus, $I_{k-2}^3$ are partitioned into classes $C_1', \dots, C_m'$ for $m\ge 3$ where each class has the form $N(\bfv)\cap I_{k-2}^3$ for some (maybe not unique) $\bfv\in I_2^3$.
If $m\ge 4$, then consider the sums $\sum_{(j_1, j_2, j_3)\in C_i'} \binom{|V_1|} {j_1} \binom{|V_2|}{j_2} \binom{|V_3|}{j_3}$ for $i=1,\dots, m$ and note that their sum is at most $\binom{n - |V_0|}{k-2}$.
Similar as the previous cases, by averaging, we have
\[
\delta_2(H) \le \frac1{m}\binom{n - |V_0|}{k-2} + |I_{k-2}^3|\mu n^k/\binom{\e^2 n}{2}  + |V_0| n^{k-3}< (1/4 + \r)\binom{n - 2}{k-2},
\]
by $\mu \ll \e \ll \r$, a contradiction.
Otherwise, $m=3$. Since $N((1,1,0))$, $N((1,0,1))$ and $N((0,1,1))$ must be in different classes, we know that
\[
N((1,1,0)) = N((0,0,2)), \, N((1,0,1)) = N((0,2,0)), \text{ and } N((0,1,1)) = N((2,0,0)),
\]
which implies that $(1,1,-2), (1,-2,1), (-2,1,1)\in L_{\cP}^{\mu}(H)$.
\end{proof}

\medskip
\begin{proof}[Proof of Proposition~\ref{prop:lattice}]
Given such a $k$-graph $H$ and a partition $\cP$. The conclusion is trivial if $r=1$. So we may assume that $r=2$ or 3. We first apply Claim~\ref{clm:vec_lat} and conclude that $(2,-2)\in L_{\cP}^{\mu}(H)$ (for $r=2$) or $(-2,1,1), (1,-2,1), (1,1,-2)\in L_{\cP}^{\mu}(H)$ (for $r=3$).

If $r=2$, then fix any $U\subseteq V(H)\setminus V_0$ with $\bfi_{\cP}(U) = (a, k+1-a)$ for some $0\le a\le k+1$ and pick any $(a_0, b_0)\in I_{\cP}^{\mu}(H)$. Since $(2,-2)\in L_{\cP}^{\mu}(H)$, then $(a_0 + 2i, b_0 - 2i)\in L_{\cP}^{\mu}(H)$ for any integer $i$.
Note that $a-1$ and $a$ have different parities, so exactly one of $(a-1, k+1-a)$ and $(a, k-a)$ is in $L_{\cP}^{\mu}(H)$.

Now assume that $r=3$ and consider any $U\subseteq V(H)\setminus V_0$ with $\bfi_{\cP}(U) = (a_1, a_2, a_3)$ for some nonnegative integers $a_1+ a_2 + a_3= k+1$.
Pick any $(b_1, b_2, b_3)\in I_{\cP}^{\mu}(H)$ and let $c_j = a_j - b_j$ for $j\in [3]$.
Note that exactly one of the three consecutive integers $c_3 - c_2-1, c_3 - c_2$ and $c_3 - c_2 +1$ is divisible by 3. Thus let $i\in [3]$ such that $\bfv: = (c_1', c_2', c_3') =(c_1, c_2, c_3) - \bfu_i$ satisfies that $c_3' - c_2'$ is divisible by 3.
Let $m = (c_3' - c_2')/3$ and $m' = m+ c_2'$. 
Note that $c_1' + c_2' + c_3' =0$ and it is easy to see that
\[
\bfi_{\cP}(U) - \bfu_i - (b_1, b_2, b_3) = \bfv = m'(-2, 1, 1) - m(1,1,-2)\in L_{\cP}^{\mu}(H).
\]
Thus, $\bfi_{\cP}(U) - \bfu_i\in L_{\cP}^{\mu}(H)$ and the proof is complete.
\end{proof}

\section{Proof of Theorem~\ref{thm:npm3}}

We prove Theorem~\ref{thm:npm3} in this section. Note that $1-(1-\frac1k)^{k-d} = \frac{2k-1}{k^2}$ when $d=k-2$.
For $k=4,5$, $\frac{2k-1}{k^2}>1/3>g(k,k-2,1)$ and thus for this range, the result is covered by Corollary~\ref{cor:npm}.
For the cases $k\ge 6$, we will use the absorbing method adapted from \cite{RRS09}.

Given a set $S$ of $k+2$ vertices, we call an edge $e\in E(H)$ disjoint from $S$ \emph{$S$-absorbing} if there are two disjoint edges $e_1$ and $e_2$ in $E(H)$ such that $|e_1\cap S|=k - 2$, $|e_1\cap e| = 2$, $|e_2\cap S|=4$, and $|e_2\cap e|=k-4$. 
Note that this is not the absorbing in the usual sense because $e_1\cup e_2$ misses two vertices of $S\cup e$.
Let us explain how such absorbing works. Let $S$ be a $(k+2)$-set and $M$ be a matching, where $V(M)\cap S=\emptyset$, which contains an $S$-absorbing edge $e$. Then $M$ can ``absorb'' $S$ by replacing $e$ in $M$ by $e_1$ and $e_2$ (two vertices of $e$ become uncovered).

\begin{lemma}\label{lem:absk2}
For all $k\ge 6, c > 0$ there exists $\beta_0>0$ such that the following holds for all $0<\beta\le \beta_0$ and sufficiently large integer $n$. 
Let $H$ be an $n$-vertex $k$-graph with $\delta_{k-2}(H)\ge c n^2$, then there exists a matching $M'$ in $H$ of size $|M'|\le \beta n$ such that for every $(k+2)$-tuple $S$ of vertices of $H$, the number of $S$-absorbing edges in $M'$ is at least $\beta^2 n$.
\end{lemma}

\begin{proof}
Our proof is adapted from the proofs of \cite[Fact 2.2, Fact 2.3]{RRS09}.
Let $\beta_0 = c^3/(12k!)$ and $0<\beta \le \beta_0$.
Let $H$ be an $n$-vertex $k$-graph with $n$ sufficiently large and $\delta_{k-2}(H)\ge c n^2$.
Given any $(k+2)$-set of vertices $S$, we will show that there are many $S$-absorbing edges.
Let us fix four vertices $u_1, \dots, u_4$ in $S$ and count only those $S$-absorbing edges $e$ for which the corresponding edge $e_2$ contains $u_1,\dots, u_4$.
We count the ordered $k$-tuples of distinct vertices $(v_1,\dots, v_k)$ such that $e=\{v_1, \dots, v_k\}$ is disjoint from $S$, $e_1\cap e = \{v_{k-3}, v_{k-2}\}$ and $e_2 = \{v_1,\dots, v_{k-4}, u_1,\dots, u_4\}$, and divide the result by $k!$.

For each $j = 1,\dots, k-6$, there are precisely $n-j-k$ choices of vertex $v_j$. Having selected $v_1,\dots, v_{k-6}$, each of $\{v_{k-5}, v_{k-4}\}$, $\{v_{k-3}, v_{k-2}\}$ and $\{v_{k-1}, v_{k}\}$ must be a neighbor of an already fixed $(k-2)$-tuple of vertices.
Thus, there are at least $\delta_{k-2}(H) - 2k n$ choices for each pair.
Altogether since $n$ is large enough, there are at least $(n - 2k)^{k-6} (\delta_{k-2}(H) - 2k n)^3 \ge \frac12 c^3 n^k$ choices of the desired ordered $k$-tuples.
So there are at least $\frac12 c^3 n^k/k!$ $S$-absorbing edges in $H$.

Now we pick the absorbing matching $M'$. Select a random subset $M$ of $E(H)$, where each edge is chosen independently with probability $p=\beta n^{1-k}$.
Then, the expected size of $M$ is at most $\binom nk p < \beta n/k!$, and the expected number of intersecting pairs of edges in $M$ is at most $n^{2k-1}p^2 = \beta^2 n$.
Hence, by Markov's inequality, with probability at least $1-1/2-1/k!$, $|M|\le \beta n$ and $M$ contains at most $2\beta^2 n$ intersecting pairs of edges.
Moreover, for every $(k+2)$-set of vertices $S$, let $X_S$ be the number of $S$-absorbing edges in $M$. Then we have
\[
\mathbb{E} (X_S) \ge p\cdot \frac12 c^3 n^k/k! \ge \frac{\beta c^3 n}{2k!}.
\]
By Chernoff's bound, with probability $1-o(1)$, we have that $X_S \ge \frac12 \mathbb{E} (X_S)\ge \frac{\beta c^3 n}{4k!}$ for all $(k+2)$-sets $S$ in $H$.

Thus, there is an $M\subseteq E(H)$ satisfying all the properties above. We delete one edge from each intersecting pair of edges and denote the resulting matching by $M'$. 
So $|M'|\le \beta n$ and for every $(k+2)$-set of vertices $S$, $M'$ contains at least $\frac{\beta c^3 n}{4k!} - 2\beta^2 n \ge \beta^2 n$ $S$-absorbing edges, by the definition of $\beta$.
\end{proof}

\begin{proof}[Proof of Theorem~\ref{thm:npm3}]
Fix $\r>0$. We apply Lemma~\ref{lem:absk2} with $c=\frac{2k-1}{2k^2}$ and get $\beta_0$. Let $\beta = \min\{\beta_0, \r/(3k)\}$.
Let $H$ be an $n$-vertex $k$-graph such that $n\equiv \ell$ mod $k$ for some $\ell\in \{2,\dots, k-1\}$ is sufficiently large and $\delta_{k-2}(H)\ge (\frac{2k-1}{k^2} + \r)\binom{n-k+2}{2}$.
We apply Lemma~\ref{lem:absk2} and get the absorbing matching $M'$ of size at most $\beta n$ and satisfying the absorbing property.

Let $H' = H[V(H)\setminus V(M')]$ and $n' = |V(H')|$. Note that
\[
\delta_{k-2}(H') \ge \left( \frac{2k-1}{k^2} + \r \right)\binom{n-k+2}{2} - k \beta n (n-1) \ge \left(\frac{2k-1}{k^2} + \r/2 \right)\binom{n'-k+2}{2}.
\]
Thus, by Theorem~\ref{thm:KOT}, $H'$ contains a matching $M_1$ that leaves at most $\beta^2 n$ vertices uncovered.
Fix any $(k+2)$-tuple of uncovered vertices $S$, $M'$ contains at least $\beta^2 n$ $S$-absorbing edges. Fix an $S$-absorbing edge $e$, we replace $M'$ by $M_S':= (M'\setminus \{e\} ) \cup \{e_1, e_2\}$, decreasing the number of uncovered vertices by $k$.
Since we have at most $\beta^2 n/k$ iterations, there will always be an available $S$-absorbing edge in $M'$. In the end, we have exactly $\ell$ vertices left uncovered and we are done.
\end{proof}

%\section{Concluding Remarks}
%
%In this note, we show that there are divisibility barriers that prevent the existence of near perfect matchings in $k$-graphs, and thus the thresholds are expected to be achieved by the maximum of the space barriers and the divisibility barriers, assembling the case of perfect matchings.
%

\section{Concluding Remarks}
We remark that Proposition~\ref{prop:lattice} is the bottle neck in the proof of Theorem~\ref{thm:npm2}.
Indeed, in the current proof, by the assumption $\delta_d(H)\ge (g(k,d,1)+o(1))\binom{n-d}{k-d} \ge (1/4+o(1))\binom{n-d}{k-d}$, it suffices to study vertex partitions $\cP=\{V_0,\dots, V_r\}$ with $r\le 3$. 
So, to improve Theorem~\ref{thm:npm2}, one has to analyze the vertex partition with more parts and prove a stronger version of Proposition~\ref{prop:lattice}.
%Moreover, we do not know a good lower bound on $g(k,d,\ell)$ in general.

\section*{Acknowledgement}
The author is indebted to Yi Zhao and anonymous referees for careful readings and comments.
The author would like to thank Wei Gao, Yoshiharu Kohayakawa and Yi Zhao for helpful discussions.

\bibliographystyle{amsplain}
\bibliography{Apr2015}

\end{document}